\documentclass{amsart}
\usepackage{mathptmx}
\usepackage{amsthm}[11pt]

 \newtheorem{thm}{Theorem}[section]
   \newtheorem{lem}[thm]{Lemma}
   \newtheorem{prop}[thm]{Proposition}
   \newtheorem{cor}[thm]{Corollary}
   \theoremstyle{definition}
   \newtheorem{dfn}[thm]{Definition}
   \newtheorem{exm}[thm]{Example}
   \newtheorem{rmk}[thm]{Remark}
   
\newtheorem*{theorem*}{Theorem}

   \DeclareMathOperator{\DAtom}{DAtom}
   \DeclareMathOperator{\Atom}{Atom}
   
   \DeclareMathOperator{\Sub}{Sub}

   \DeclareMathOperator{\girth}{girth}
   \DeclareMathOperator{\diam}{diam}
   \DeclareMathOperator{\Id}{Id}

\begin{document}

\title{On graphs of bounded semilattices}


\author{Parastoo Malakooti Rad}
\address{P. Malakooti Rad, Department of Mathematics, Qazvin Branch, Islamic Azad University, Qazvin, Iran}
\email{pmalakoti@gmail.com}

\author{Peyman Nasehpour}
\address{P. Nasehpour, Department of Engineering Science, Golpayegan University of Technology, Golpayegan, Iran}
\email{nasehpour@gut.ac.ir, nasehpour@gmail.com}

 \keywords{Intersection Graphs, Bounded Semilattices, Eulerian Graph, Planar Graph}
 \subjclass[2000]{05C99; 06A12}
 \maketitle

\begin{abstract}
In this paper, we introduce the graph $G(S)$ of a bounded semilattice $S$, which is a generalization of the intersection graph of the substructures of an algebraic structure. We prove some general theorems about these graphs; as an example, we show that  if $S$ is a product of three or more chains, then $G(S)$ is Eulerian if and only if either the length of every chain is even or all the chains are of length one. We also show that if $G(S)$ contains a cycle, then $\girth(G(S)) = 3$. Finally, we show that if $(S,+,\cdot,0,1)$ is a dually atomic bounded distributive lattice 
whose set of dual atoms is nonempty, 
 and the graph $G(S)$ of $S$ has no isolated vertex, then $G(S)$ is connected with $\diam(G(S))\leq 4$.
 \end{abstract}

\section{Introduction}\label{sec:intro}

The partially ordered set $(S, \leq)$ is called a \emph{meet-semilattice} if every two elements $x$ and $y$ of $S$ have a greatest lower bound $x \wedge y \in S$. Equivalently, for  a binary operation $\wedge$ on $S$, the structure  $(S, \wedge)$ is a \emph{meet-semilattice} if $\wedge$ is associative, commutative, and idempotent (i.e., a commutative idempotent semigroup). We denote the smallest element of a meet-semilattice by $0$, and the largest element by $1$. It is called \emph{bounded} if it has a smallest and a largest element. A \emph{join-semilattice} is defined dually, and a \emph{bounded semilattice} will be a meet or join-semilattice with both a $0$ and a $1$.

Given a bounded semilattice $(S, \circ, 0, 1)$, define a graph $G(S)$ as follows:

\begin{enumerate}

\item The set of vertices $V$ of $G(S)$ is the set of all elements of $S$ except $0$ and $1$.

 \item The vertices $x,y \in V$ are adjacent, i.e. $\{x,y\}$ belongs to the edges $E$ of $G(S)$, if $x \neq y$ and $x\circ y \neq 0$.

 \end{enumerate}

We need two more definitions. Let $(S, \circ, 0, 1)$ be a bounded semilattice, then an \emph{atom} is a minimal element of $S-\{0,1\}$, and $S$ will be called \emph{Artinian} if every decreasing chain of elements becomes stationary.

In this paper, we initiate the study of the graph $G(S)$ for general bounded semilattices $S$, and, for example, we prove the following new results:

\medskip

\begin{itemize}\itemindent=5em
\item[{\bf Theorem \ref{K2K12}}:] If $G(S)$ is a path of length $k$, then either $G(S)=K_2$, or $G(S)=K_{1,2}$.
\item[{\bf Proposition \ref{Artinianprop}}:] If $S$ is Artinian with more than two elements, then $G(S)$ is a complete graph if and only if $S$ has exactly one atom.
\item[{\bf Theorem \ref{Euler2}}:] If $S$ is a product of two or more chains, then $G(S)$ is Eulerian if and only if either the length of every chain is even or if all the chains are of length one.
\item[{\bf Theorem \ref{TreeStar}}:] If $G(S)$ is a tree, then it is a star graph.
\item[{\bf  Proposition \ref{MinDiam}}:] If $S$ has more than three elements and exactly one atom, then $G(S)$ is a complete graph.
\item[{\bf Theorem \ref{girth}}:] If $G(S)$ contains a cycle, then its girth is equal to $3$.
\end{itemize}

\medskip

Intersection graph theory (for short IGT) is a classical topic in the theory of graphs \cite{ErdosGoodmanPosa1966}. For a good introduction to IGT, one can refer to the book \cite{McKeeMcMorris1999}. And in this classic book, some applications of IGT in different fields of science such as biology, psychology, and computing are mentioned in details \cite[\S 2 and \S 3]{McKeeMcMorris1999}. Although all graphs are intersection graphs \cite{Szpilrajn1945}, some classes of intersection graphs are of special interest. For example, the intersection graphs of some classes of geometrical objects, e.g. closed intervals of the real-line (see \cite[p. 1]{Cohen1977}, \cite{Cohen1978} and \cite[p. 43]{CohenBriandNewman1990}), chords of a circle \cite[p. 137]{Sherwani2002}, trapezoids between two horizontal lines \cite{LiangLuTang1997}, and unit disks in plane \cite{LotkerPeleg2010} have interesting applications in science and industry.

On the other hand, the intersection graphs of substructures of an algebraic structure have been investigated by many authors \cite{AfkhamiKhashyarmanesh2014, Bosak1964, ChakrabartyGhoshMukherjeeSen2009, CsakanyPollak1969, Osba2016, Shen2010, Zelinka1975, Zelinka1973}. Our original motivation for this work was the intersection graphs of submodules of a module \cite{AkbariTavallaeeGhezelahmad2012} and our discussions on this topic led us to work on a more general context, i.e. graphs that we attributed to bounded semilattices.

\section{The Graphs of Bounded Semilattices}\label{sec:bsl}

Let us recall that $(S,\circ)$ is called a semilattice, if $(S,\circ)$ is a commutative semigroup and its binary operation $\circ$ is idempotent, i.e. $x \circ x= x$, for all $x\in S$ \cite[Definition 2.1.1]{ChajdaHalasKuhr2007}. It is good to mention that a similar definition for semilattices is given in \cite[Section 4.1]{Vickers1989}. It is easy to see that a partial order is induced on the semilattice $S$ by setting $x \leq y$ whenever $x\circ y = x$, for all $x,y \in S$ \cite[Theorem 2.1.2]{ChajdaHalasKuhr2007}. Finally, note that if 1 is the neutral element of $S$, then $x \leq 1$, for all $x\in S$. And if $0$ is an absorbing element of a semilattice $S$, that is, $x\circ 0 = 0$, for all $x\in S$, then $0$ is the least element of $S$, i.e. $0 \leq x$, for all $x\in S$. If the semilattice $S$ possesses neutral and absorbing elements, then $S$ is called bounded, since $0 \leq x\leq 1$ for all $x\in S$. One of the simplest semilattices that may come to one's mind is the semilattice $(\mathbb{P}(A), \cap)$, where by $\mathbb{P}(A)$  we mean the set of all subsets of the set $A$.

One can easily check if $A$ is a set and $\mathcal{A} \subseteq\mathbb{P}(A)$, then $(\mathcal{A}, \cap)$ is a bounded semilattice if and only if the following properties hold:

\begin{enumerate}

\item If $X, Y \in \mathcal{A}$, then $X \cap Y \in \mathcal{A}$,

\item There are two distinct sets $M_1$ and $M_2$ in $\mathcal{A}$ such that $M_1 \subseteq X \subseteq M_2$, for all $X \in  \mathcal{A}$.

\end{enumerate}

In this paper, the semilattice $\mathbb{P}(S)$ is of special interest, when $S$ has an algebraic structure since it provides some good examples for our results. Now, we attribute a graph to a bounded semilattice, inspired by the definition of intersection graphs in \cite{McKeeMcMorris1999}.

\begin{dfn}

\label{IntersectionGraphSemiLattice}

Let $(S,\circ,0,1)$ be a bounded semilattice. We attribute a graph $G(S)$ to $S$, whose vertices $V$ and edges $E$ are determined as follows:

\begin{enumerate}

\item The set of vertices $V$ is the set of all elements of $S$ except $0$ and $1$.

 \item The vertices $x,y \in V$ are adjacent, i.e. $\{x,y\} \in E$, if $x \neq y$ and $x\wedge y \neq 0$.
\hfill $\diamond$ 
 \end{enumerate}
\end{dfn}

The following remark justifies why our definition for graphs of semilattices given in Definition \ref{IntersectionGraphSemiLattice} is a generalization of the intersection graphs of substructures of different algebraic structures.

\begin{rmk}[Intersection Graphs of Algebraic Structures]\label{IntersectionGraphEx}

$ $

\begin{enumerate}
\item Let $S$ be a semigroup and $\mathcal{S}$ the set of all subsemigroups of $S$. Clearly, the structure $(\mathcal{S} \cup \{\emptyset\}, \cap)$ is a bounded semilattice and its graph $G(\mathcal{S})$, given in Definition \ref{IntersectionGraphSemiLattice} of the current paper, coincides with the definition of the graphs of semigroups introduced in \cite{Bosak1964}.

\item Let $R$ be a commutative ring with a nonzero identity and $M$ a unitary nonzero $R$-module. It is obvious that the intersection graph of an $R$-module $M$, introduced in \cite{AkbariTavallaeeGhezelahmad2012}, is just the graph $G(\Sub_R(M))$ of the bounded semilattice $\Sub_R(M)$, where by $\Sub_R(M)$, we mean the set of all $R$-submodules of $M$. For more results on the intersection graph of a module, one may also refer to \cite{Yaraneri2013}.

\item Let $S$ be a semiring and $M$ an $S$-semimodule. It is easy to see that $(\Sub_S(M), \cap)$ is a bounded semilattice, where by $\Sub_S(M)$, we mean the set of all $S$-subsemimodules of $M$. In some cases, we will investigate the intersection graph $G(M)$ of the subsemimodules of the $S$-semimodule $M$.

\item Other examples for bounded semilattices and their intersection graphs include subgroups of a group \cite{CsakanyPollak1969,Zelinka1975}, normal subgroups of a nontrivial group, left ideals of a semiring \cite[\S 6]{Golan1999}, left ideals of the ring possessing a nonzero identity \cite{ChakrabartyGhoshMukherjeeSen2009}, subsemirings of the semiring \cite{LinRatti1970}, subsemimodules of a nonzero semimodule \cite[\S 14]{Golan1999}, and clopen sets of a topology, where by a clopen set, it is meant a set that is both closed and open \cite[Definition 3.6.4]{Vickers1989}.
\hfill $\diamond$ 
\end{enumerate}

\end{rmk}

\begin{dfn}

\label{atomdef}

Let $S$ be a bounded semilattice. An element $a\in S$ is called to be an atom, if $0 < a < 1$ and also, if $0 \leq y \leq a$, then either $y=0$ or $y = a$. We gather atoms of $S$ in the set $\Atom(S)$. Also, an element $d\in S$ is called to be a dual atom, if $0 < d < 1$ and also, if $d \leq y \leq 1$, then either $y=d$, or $y=1$. We gather dual atoms of $S$ in $\DAtom(S)$.
\hfill $\diamond$ 
\end{dfn}

\begin{rmk}
Let $S$ be a bounded semilattice. It is clear that atoms of $S$ are the minimal elements of the poset $S-\{0,1\}$, and dual atoms of $S$ the maximal elements of the poset $S-\{0,1\}$. Note that if $\mathcal{S}$ is the semilattice of the ideals of a commutative semiring $R$, then the dual atoms of $\mathcal{S}$ are nothing, but the maximal ideals of $R$.
\hfill $\diamond$ \end{rmk}

Let us recall that the degree of a vertex $v$ in a graph $G$, denoted by $d(v)$, is the number of edges of $G$ incident with $v$ \cite[p. 7]{BondyMurty2008}.

\begin{prop}

\label{deg1}

Let $S$ be a bounded semilattice and the graph $G(S)$ have no cycle of length $3$. If $y\in \Atom(S)$, then $\deg(y) =1$.
\end{prop}

\begin{proof}
Let $y\in \Atom(S)$, but $\deg(y) \geq 2$. So, there exist at least two distinct vertices  $y_1$ and $y_2$ of $G(S)$ such that both  are adjacent to $y$. Therefore, $y y_1\neq 0$ and $y y_2 \neq 0$. Since $y\in \Atom(S)$, $y y_1= y= y y_2$. Hence, $y \leq y_1$ and $y\leq y_2$. Thus $0 \neq y=y^2 \leq y_1 y_2$  and this implies that $y_1$ and $y_2$ are adjacent. Thus $G(S)$ contains a cycle  $y-y_1-y_2-y$, which is a contradiction. Consequently, $\deg(y) =1$.
\end{proof}

\begin{prop}

\label{minimax}

Let $S$ be a bounded semilattice and $y$ a vertex of $G(S$). If $\deg (y)=1$, then either $y \in \Atom(S)$ or $y \in \DAtom(S)$.
\end{prop}
\begin{proof}
  Let $y$ be a vertex of $G(S$) such that $\deg (y)=1$  and $z$ be the only vertex of $G(S)$ such that $z$ is adjacent to $y$. Clearly, $y z \neq 0$. Our claim is that either $yz=z$ or $yz=y$. Suppose that $yz \neq y$. Therefore, $y\cdot yz = yz \neq 0$, which means that $yz$ is adjacent to $y$ and this implies that $yz = z$. So we have showed that either $y \leq z$ or $z \leq y$. If $y \leq z$, then there is no nonzero element $l\in S$ such that $l < y$. So, $y$ is in $\Atom(S)$. If $z \leq y$, then there is no $m \in S-\{1\}$ such that $y < m$. So, $y$ is in $\DAtom(S)$ and the proof is complete.
\end{proof}

\begin{rmk} The converse of Proposition \ref{minimax} does not hold. For example, let $(R, \mathfrak{m})$ be a quasi-local semiring, i.e. a semiring with the unique maximal ideal $\mathfrak{m}$. Clearly, if $|\Id(R)| \geq 5$, then $\deg(\mathfrak{m})\geq 2$. Note that any valuation semiring is quasi-local \cite[Theorem 1.8]{Nasehpour2017}.
\hfill $\diamond$ 
\end{rmk}

Let us recall that a path is a simple graph whose vertices can be arranged in a linear sequence in such a way that two vertices are adjacent if they are consecutive in the sequence, and are nonadjacent otherwise \cite[p. 16]{BondyMurty2008}.

\begin{lem}

\label{MaxK2}

Let $S$ be a bounded semilattice and $G(S)$ a path as sequence $$y_1, y_2,\ldots,y_t,$$ where $t\geq2$. If $y_1\in \DAtom(S)$, then $G(S)= K_2$, where   $K_2$  is  the complete graph on two vertices.
\end{lem}
\begin{proof}
Let $G(S)$ be a path as sequence $y_1, y_2,\ldots,y_t$, where $t\geq2$ and $y_1\in \DAtom(S)$. Then, either $y_1 y_2=y_1$ or the vertex $y_1 y_2$ is adjacent to $y_1$. If $y_1 y_2=y_1$, then $y_1 \leq y_2$. This implies that $y_1 = y_2$, since  $y_1 \in \DAtom(S)$, and obviously, this is a contradiction, since $y_1$ and $ y_2$ are distinct vertices of $G(S)$. Since by assumption, the only vertex adjacent to $y_1$ is the vertex $y_2$, $y_1 y_2=y_2$ and this means that $y_2 \leq y_1$. Now, we prove that $t$ cannot be greater than 2. In contrary, let $t\geq 3$. Therefore, either $0\neq y_2  y_3=y_2$ or the vertex $y_2 y_3$ is adjacent to $y_2$. If $0\neq y_2 y_3=y_2$, then $y_2\leq y_3$ and so $y_2\leq y_3 y_1$. This means that $y_3 y_1 \neq 0$ and so the vertices $y_1$ and $y_3$ are adjacent, which is a contradiction. But the only vertices that are adjacent to $y_2$ are $y_1$ and $y_3$. So, either $y_2 y_3= y_1$ or $y_2 y_3= y_3$. If  $y_2 y_3= y_1$, then $y_1$ and $y_3$ are adjacent, which is a contradiction. Otherwise,  $y_2 y_3= y_3$ and this implies that $y_3\leq y_2$. Now in view of $y_2\leq y_1$, we get that the vertices $y_1$ and $y_3$ are adjacent, again a contradiction. Hence, $G(S)=K_2$ and the proof is complete.
 \end{proof}

 \begin{cor}
Let $S$ be a bounded semilattice such that $G(S)$ is a path. Then $G(S)=K_2$ if and only if $|\DAtom(S)|=|\Atom(S)|=1$.
\end{cor}
\begin{proof}
($\Rightarrow$): Let $G(S)=K_2$. So, by definition, $G(S)$ has only two vertices $y_1$ and $y_2$ and they are adjacent, which means that $y_1 y_2 \neq 0$. Obviously, this implies that either $y_1 y_2=y_1$ or $y_1 y_2=y_2$. If $y_1 y_2=y_1$, then $y_1\leq y_2$. This implies that $y_1$ is in $\Atom(S)$ and $y_2$ is in $\DAtom(S)$. Similarly, if $y_1 y_2=y_2$, then  $y_2$ is in $\Atom(S)$ and $y_1$ is in $\DAtom(S)$ and therefore, in each case, $|\DAtom(S)|=|\Atom(S)|=1$.

($\Leftarrow$): Straightforward.
\end{proof}

\begin{thm}

\label{K2K12}

Let $S$ be a bounded semilattice and $G(S)$ a path. Then, either $G(S)=K_2$, or $G(S)=K_{1,2}$.
\end{thm}
\begin{proof}
Let $G(S)$ be a path as sequence $y_1,y_2,\ldots,y_n$. By Proposition \ref{minimax}, the element $y_1$ is either in $\Atom(S)$ or in $\DAtom(S)$. If $y_1 \in \DAtom(S)$, then by Lemma \ref{MaxK2}, $G(S) = K_2$.

Therefore, let, for the moment, $y_1$ be in $\Atom(S)$. Since vertices $y_1$ and $y_2$ are adjacent, we have $y_1 y_2\neq 0$. But $y_1  y_2 \leq y_1$ and $y_1$ is in $\Atom(S)$. So, $y_1\leq y_2$. Now, we prove that if $n=3$, then $G(S)=K_{1,2}$. So, let $G(S)$ be a path as a sequence $y_1,y_2,y_3$ such that $y_1 \in \Atom(S)$. First of all, if $y_2 y_3= y_2$, then $y_2\leq y_3$ and so, we conclude that $y_1$ and $y_3$ are adjacent, a contradiction. Hence, either $y_2 y_3= y_1$ or  $y_2 y_3= y_3$. If $y_2 y_3= y_1$, then $y_1$ and $y_3$ are adjacent, a contradiction. Thus, $y_2 y_3= y_3$ and so $y_3\leq y_2$.

Now, if we prove that $n$ cannot be greater than 3, we are done. In contrary, let $n>3$. Vividly, since $y_3 y_2 \neq 0$, we have $y_3 y_2 = y_t$, for some $1\leq t \leq n$. If $t > 3$, then $y_t \leq y_2$ and so, $y_2$ and $y_t$ are adjacent, a contradiction. Hence, $n$ cannot be greater than 3, i.e. $G(S) = K_{1,2}$ and the proof is complete.
\end{proof}

\begin{cor}
Let $(L,+,\cdot,0,1)$ be a bounded lattice and $G(L)$ be a path. Then, either $G(L)=K_2$ or $G(L)=K_{1,2}$. Moreover, if  $G(L)=K_{1,2}$, then $G(L)$ is of the form $y_1 - y_2 - y_3$ with $y_2=y_1 + y_3$ and $y_1  y_3 = 0$.
\end{cor}
\begin{proof}
As we have seen in the proof of Theorem \ref{K2K12}, we have $y_3 \leq y_2$ and $y_1 \leq y_2$. This implies that $y_3 + y_1 \leq y_2$. So, $y_3 + y_1 = y_t$ for some $1 \leq t \leq 3$. If $y_3 + y_1 = y_1$, then $y_3 \leq y_1$, which implies that $y_3$ is adjacent to $y_1$, a contradiction. In a similar way, one can see that $y_3 + y_1 \neq y_3$. So, $y_3 + y_1 = y_2$. Note that since, $y_3$ is not adjacent to $y_1$, we have $y_1  y_3 = 0$.
\end{proof}

One of the corollaries of Theorem \ref{K2K12} is the following result for semimodules. For the definition of semirings and semimodules, one can refer to the book \cite{Golan1999}.

\begin{cor}
Let $S$ be a semiring and $M$ an $S$-semimodule. If the intersection graph $G(M)$ of the $S$-subsemimodules of $M$ is a path, then either $G(M)=K_2$, or $G(M)=K_{1,2}$. Moreover, if  $G(M)=K_{1,2}$, then $G(M)$ is of the form $N_1 - N_2 - N_3$ with $N_2=N_1 + N_3$ and $N_1 \cap N_3 = (0)$, where $N_1, N_2, N_3$ are $S$-subsemimodules of $M$.
\end{cor}

Let us recall that in a commutative semigroup $S$ with zero, $s\in S$ is a zero-divisor if there is a nonzero $t\in S$ such that $st =0$.

\begin{prop}
	
	\label{Complete}
	
	Let $S$ be a bounded semilattice with more than two elements. Then $G(S)$ is complete if and only if $S$ has no zero-divisors other than 0.
	
	\begin{proof}
		Straightforward.
	\end{proof}
	
\end{prop}

\begin{dfn}

\label{Artiniandef}

A bounded semilattice $S$ is Artinian, if any decreasing chain $$s_1 \geq s_2 \geq \cdots s_n \geq s_{n+1} \geq \cdots$$ in $S$ is stationary, i.e. there is an $n\in \mathbb N$ such that $s_i = s_{i+1}$, for all $i \geq n$.
\hfill $\diamond$ 
\end{dfn}

\begin{cor}

\label{Artinianprop}

Let $S$ be an Artinian bounded semilattice with more than two elements. Then, $G(S)$ is complete if and only if $|\Atom(S)| = 1$.
\end{cor}
\begin{proof}
Let $S$ be an Artinian bounded semilattice with more than two elements. Clearly, $\Atom(S) \neq \emptyset$.

($\Leftarrow$): Let $|\Atom(S)| = 1$ and $m$ be the unique element of $\Atom(S)$. Clearly, $m \leq x$, for all $x\in S-\{0,1\}$. This implies that if $x,y \in S-\{0,1\}$, then $xy \geq m$. So, $xy\neq 0$, for all $x,y \in S-\{0,1\}$, which means that $G(S)$ has no zero-divisor other than 0 and therefore, by Proposition \ref{Complete}, it is complete.

($\Rightarrow$): If $m_1$ and $m_2$ are two distinct elements of $\Atom(S)$, then $m_1 m_2 = 0$. So, $S$ has some zero-divisors other than 0. Hence, by Proposition \ref{Complete}, $G(S)$ is not complete.
\end{proof}

Let us recall that if $S$ is a poset, then the length of $S$, denoted by $l(S)$, is defined as $l(S) = \sup \{ |C|-1: C \text{ is a chain of $S$}\}$ \cite[p. 54]{Blyth2005}.

A graph is said to be planar if it can be drawn in the plane so that its edges intersect only at their ends. Kuratowski's Theorem in graph theory states that a graph is planar if and only if it contains no subdivision of either $K_5$ or $K_{3,3}$ \cite[Theorem 10.30]{BondyMurty2008}.

\begin{prop}
\label{planarchain}
 Let $S$ be a bounded semilattice. If $G(S)$ is a planar graph, then $l(S) \leq 5$.
 \end{prop}
\begin{proof}
Let $l(S) \geq 6$. So there exists a chain $0 < s_1 < s_2 < s_3 < s_4 < s_5 <1$ in $S$ such that $s_i \in  S-\{0,1\}$. Clearly, the
vertices $s_i$, where $1\leq i\leq 5$, form $K_5$ as an induced subgraph of $G(S)$, a contradiction. So, $l(S) \leq 5$ and the proof is complete.
\end{proof}

Let us recall that a walk in a graph $G$ is a sequence $v_0 e_1 v_1 \cdots v_{l-1}e_l v_l$, whose terms are alternately vertices and edges of $G$, such that $v_{i-1}$ and $v_i$ are the ends of $e_i$, $1 \leq i \leq l$. A walk in a graph is closed if its initial and terminal vertices are identical. A tour of a connected graph $G$ is a closed walk that traverses each edge of $G$ at
least once, and a Euler tour one that traverses each edge exactly once. A graph is Eulerian if it admits a Euler tour (see Sections 3.1 and 3.3 in \cite{BondyMurty2008}). A graph in which each vertex has even degree is called an even graph. A connected graph is Eulerian if and only if it is even \cite[Theorem 3.5]{BondyMurty2008}.

\begin{lem}

\label{Euler1}

Let $S$ be a finite bounded chain (semilattice) with more than two elements. Then $G(S)$ is a complete graph. Moreover, $G(S)$ is Eulerian if and only if $l(S)$ is an even number.
\end{lem}
\begin{proof}
 Let $l(S) = t+1$ and set $S=\{0,s_1,\ldots,s_t,1\}$ such that $0 < s_1 < \cdots < s_t <1$. It is clear that $s_i s_j = s_{\min\{i,j\}} \neq 0$. So, $G(S)$ is the complete graph $K_t$ and $\deg(s_i) = t-1$, for each $i$. Therefore, $G(S)$ is Eulerian if and only if $l(S)=t+1$ is even.
\end{proof}

It is easy to verify that if $\{S_i\}$ is a family of bounded semilattices, then $S = \prod_i S_i$ is also a bounded semilattice, where its operation is defined componentwise and $1_S = (1_{S_i})$ and $0_S = (0_{S_i})$.

\begin{thm}

\label{Euler2}

Let $n\geq 3$ and $\{S_i\}^n_{i=1}$ be a family of bounded semilattices. If each $S_i$ is a finite chain and $S= \prod^n_{i=1} S_i$, then $G(S)$ is Eulerian if and only if either (a) the length $l(S_i)$ of $S_i$ is even, for all $1 \leq i \leq n$ or (b) each $S_i$ has two elements.

\end{thm}

\begin{proof}

Let $x=(x_1, \ldots, x_n) \in S-\{0_S,1_S\}$. Define $\delta_i : S \longrightarrow \{0,1\}$ as follows:

$$\delta_i(x)=\left\{
          \begin{array}{ll}
            1, & \hbox{ $x_i= 0$;} \\
            0, & \hbox{ $x_i\neq 0$.}
          \end{array}
        \right.
$$

Therefore, the number of elements $y \in S-\{0_S, 1_S\}$ such that $x y=0$ is $$\prod_{i=1}^n(l_i+1)^{\delta_i(x)}-1,$$ where $l_i = l(S_i)$. Also, the number of vertices of the graph $G(S)$ is $$\prod^n_{i=1}(l_i +1) -2.$$ Now, since the vertex $x$ is not adjacent to $x$, $$\deg(x)=\prod_{i=1}^n(l_i+1)-\prod_{i=1}^n(l_i+1)^{\delta_i(x)}-2.$$

$(\Leftarrow)$ Proof of (b): If each $S_i$ has two elements and $x=(x_1, \ldots, x_n) \in S-\{0_S,1_S\}$, then some of the $x_i$s are 0, while the rest of the $x_i$s are 1. Now, if we set $B=\{ i: x_i =0\}$, then $1 \leq |B| < n$, and $\deg(x) = 2^n - 2^{|B|} -2$, which is clearly an even number. Therefore, $G(S)$ is Eulerian.

Proof of (a): Since $x \neq 0_S$, by symmetry, we can imagine $\delta_i = \cdots = \delta_n = 0$, for some $1\leq i \leq n$. Now, if each $l_i$ is even, then $\prod_{i=1}^n(l_i+1)$ and $\prod_{i=1}^n(l_i+1)^{\delta_i(x)}$ are both odd numbers and therefore, $\deg(x)$ is even, for each vertex $x$. This implies that $G(S)$ is Eulerian.

$(\Rightarrow)$ Now, suppose that one of the numbers $\{l_i=l_i(S_i)\}$ is odd and at the same time one of the numbers $\{l_i=l_i(S_i)\}$ is even. We define $x=(x_1,\ldots,x_n)$, where $x_i = 1$ if and only if $l_i$ is odd and $x_i = 0$ if and only if $l_i$ is even. Clearly, $\prod_{i=1}^n(l_i+1)$ is an even number, while $\prod_{i=1}^n(l_i+1)^{\delta_i(x)}$ is odd, because it is the multiplication of odd numbers. So, $\deg(x)$ is odd and $G(S)$ cannot be a Eulerian graph.

If all the numbers $l_i$ are odd and for some $i$ we have $l_i \geq 3$, then we define $y=(y_1,\ldots,y_n)$, where $y_i = 1_{S_i}$, if and only if $l_i = 1$ and $y_i$ is the unique element in $\DAtom(S_i)$ if and only if $l_i \geq 3$. Definitely, $y \neq 1_S$ and $\delta_i(y) = 0$, for each $i$. Also, $\deg(y)=  \prod_{i=1}^n(l_i+1)-3$, which is, clearly, an odd number. Therefore, $G(S)$ cannot be again a Eulerian graph and the proof is complete.
\end{proof}

In Lemma \ref{Euler1}, we proved that if $S$ is a finite bounded chain (semilattice) with more than two elements, then $G(S)$ is a complete graph. Now, we show that for a direct product of bounded chains, this is not the case.

\begin{prop}
Let $n\geq 2$ and $\{S_i\}^n_{i=1}$ be a family of bounded semilattices. If each $S_i$ is a finite chain and $S= \prod^n_{i=1} S_i$, then $G(S)$ cannot be a complete graph.
\end{prop}
\begin{proof}
Let $x=(x_1, \ldots, x_n) \in S-\{0_S,1_S\}$. Clearly, if $G(S)$ is a complete graph, then the number of elements $y \in S-\{0_S, 1_S\}$ such that $x y=0$ must be zero. Therefore, according to the proof of Theorem \ref{Euler2}, we have $$\prod_{i=1}^n(l_i+1)^{\delta_i(x)}-1 = 0.$$ Obviously, this implies that $\delta_i(x)=0$, for each $i$ and this happens only $x=1_S$, a contradiction. Therefore, $G(S)$ cannot be a complete graph and the proof is complete.
\end{proof}

Let us recall that a tree is an undirected graph in which any two vertices are connected by exactly one path \cite[Theorem 1.5.1]{Diestel2017}. A complete bipartite graph is a graph where every vertex of the first set is connected to every vertex of the second set and if one of the sets has exactly one element, it is called a star graph \cite[p. 18]{Diestel2017}.

\begin{thm}

\label{TreeStar}

Let $S$ be a bounded semilattice. Then, the graph $G(S)$ is a tree if and only if it is a star graph.

\begin{proof}
We just need to prove that if $G(S)$ is a tree, then $G(S)$ is a star graph. On contrary, let $G(S)$ be a tree such that it is not a star graph. So, $G(S)$  has a path of length 3, say of the form $y_1-y_2-y_3-y_4$ with $\deg (y_1)=1$. Since $\deg (y_1)=1$, by Proposition \ref{minimax}, either $y_1$ is in $\Atom(S)$ or $\DAtom(S)$.

Firstly, suppose that $y_1 \in \Atom(S)$. Since vertices $y_1$ and $y_2$ are adjacent, we have $y_1 y_2\neq 0$. Clearly, $y_1 y_2 \neq1$. Our claim is that $y_1 y_2 = y_1$. If $y_1 y_2 = y_2$, then $y_2 \leq y_1$ and since $y_1$ is an atom and $y_2$ is nonzero, $y_2 = y_1$, a contradiction. Now, let $y_1y_2 = s$ such that $s \in S-\{y_1,y_2\}$. In this case, $s$ is adjacent to the both vertices $y_1$ and $y_2$ and this is impossible, since $G(S)$ is a tree and any two vertices of a tree are connected by exactly one path \cite[Theorem 1.5.1]{Diestel2017}. Therefore, $y_1 y_2 = y_1$ and so, $y_1\leq y_2$.

By assumption, $y_2$ and $y_3$ are adjacent. So, $y_2 y_3 \neq 0$. Also, $y_2 y_3 \neq 1$. Now, we prove that $y_2 y_3 = y_3$.

If $y_2 y_3= y_2$, then $y_2\leq y_3$. So, we obtain that $y_1$ and $y_3$ are adjacent, a contradiction. On the other hand, if there is an element $s\in S-\{0,1\}$ such that $s$ is different from $y_2$, and $y_3$ and $y_2 y_3= s$, then $s$ is adjacent to the both vertices $y_2$ and $y_3$, again a contradiction. Therefore, $y_2 y_3= y_3$ and this implies that $y_3\leq y_2$. So, $y_2 y_4 \geq y_3y_4 \neq 0$ and this means that $y_2$ and $y_4$ are adjacent, a contradiction.

Now, suppose that $y_1 \in \DAtom(S)$. Similar to the proof in above, we can show that $y_1 y_2 = y_2$ and so, $y_2\leq y_1$. Obviously, $y_2 y_3 \neq 0,1$. On the other hand, it can be similarly proved that if $y_2y_3 = s$ for some $s\in S-\{y_2, y_3\}$, then again $G(S)$ cannot be a tree. If $y_2 y_3 = y_2$, then $y_2 \leq y_3$ and this implies that $y_1$ is adjacent to $y_3$, a contradiction. Also, if $y_2 y_3 = y_3$, then $y_3 \leq y_2$ and in this case, $y_2$ is adjacent to $y_4$, again a contradiction. Hence, if $G(S)$ is a tree, then it is a star graph and the proof is complete.
\end{proof}
\end{thm}

We end this section by a criterion for $G(L)$ being connected, where $L$ is a modular bounded lattice. Let us recall that a bounded lattice $L$ is modular if $c \leq b$ implies that $(c+a)b = c + ab$, for all $a,b,c\in L$ \cite[p. 10]{Stern1999}.

\begin{prop}

\label{modularlattice}

Let $a$ and $b$ be two distinct elements of a modular bounded lattice $L$. Then there is no path in $G(L)$ between $a$ and $b$ if and only if $ab=0$, $a+b =1$, and $a,b \in \Atom(L)$.

\begin{proof}
$(\Rightarrow)$: Assume that there is no path in $G(L)$ between $a$ and $b$. Clearly, $ab = 0$. Let $c$ be an element of $L$ such that $0 \neq c \leq b$. So, $ca=0$. On the other hand, if $c+a \neq 1$, then $a-(c+a)-b$ is a path of length 2, a contradiction. So, $c+a =1$. In particular, $a+b =1$. Now, we prove that $b\in \Atom(L)$. By modular law, $(c+a)b=c+ab$. But we have already seen that $ab =0$ and $c+a=1$. Therefore, $b=c$.
\end{proof}
\end{prop}

\begin{rmk}
Proposition \ref{modularlattice} is a special case of Theorem 2.12 in \cite{DevhareJoshiLaGrange2018}.
\hfill $\diamond$ \end{rmk}

\section{On the Diameter and Girth of the Graphs of Bounded Semilattices}\label{sec:diam}

Let us recall that the distance between two vertices in a graph is the number of edges in a shortest path connecting them. The
greatest distance between any two vertices in a graph $G$ is the diameter of $G$, denoted by $\diam(G)$ \cite[p. 8]{Diestel2017}.

\begin{prop}

\label{MinDiam}

Let $S$ be a bounded semilattice with $|S| \geq 4$ and $|\Atom(S)| = 1$. Then $G(S)$ is connected with $\diam(G(S))= 1$.

\begin{proof}

Let $m$ be the unique element of $\Atom(S)$. Therefore, for any nonzero element $s$ in $S$, we have $m \leq s$. This implies that if $s_1$ and $s_2$ are two distinct elements of $S-\{0,1\}$, then $s_1 \cdot s_2 \geq m \cdot m = m \neq 0$. So, any pair of the vertices $s_1$ and $s_2$ of $G(S)$ are connected to each other, which means that $G(S)$ is complete (connected) and $\diam(G(S))= 1$.
\end{proof}

\end{prop}

\begin{prop}

\label{MaxDiam}

Let $S$ be a bounded semilattice with $|S| \geq 3$ and $|\DAtom(S)| = 1$. Then $G(S)$ is connected with $\diam(G(S))\leq 2$.

\end{prop}
\begin{proof}
Let $m\in \DAtom(S)$. Obviously, if $y$ is a vertex of $G(S)$ distinct from $m$, then, $y\leq m$, and so, $ym = y \neq 0$. Therefore, $y$ and $m$ are adjacent. Clearly, this implies that for each vertices $x\neq m$ and $y\neq m$, we have the path $x - m - y$, which implies that the distance between any pair of vertices of $G(S)$ is at most 2 and the proof is complete.
\end{proof}

\begin{exm}
In this example, we give graphs of bounded semilattices satisfying the conditions of Proposition \ref{MaxDiam}, with diameter 0, 1, and 2. Let $k$ be a field and $X$ an indeterminate over $k$. Set $R=k[[X]]$ to be the formal power series ring over $k$. The set of all ideals of $R$ is the infinite chain $$R \supset (X) \supset (X^2) \supset \cdots \supset (X^n) \supset \cdots \supset (0).$$ Now, we set $S_n = \Id(R/(X^n))$ to be the set of all ideals of the ring $R/(X^n)$. Clearly, $S_n$ has at least three elements for any $n\geq 2$, $\DAtom(S_n) = (X)/(X^n)$, and $G(S_n) = K_{n-1}$. Therefore, $\diam(G(S_2)) = 0$ and $\diam(G(S_n))=1$, for any $n\geq 3$.

Now, let $T = k \times k[[X]]/(X^2)$. Obviously, the only maximal ideal of $T$ is the ideal $\mathfrak{n} = k \times (X)/(X^2)$. Suppose $\mathfrak{a} = k \times 0$ and $\mathfrak{b} = 0 \times k[[X]]/(X^2)$. We have $\mathfrak{a} \cap \mathfrak{b} = 0$, while $\mathfrak{a} \cap \mathfrak{n} \neq 0$ and $\mathfrak{b} \cap \mathfrak{n} \neq 0$. So, $d(\mathfrak{a}, \mathfrak{b}) = 2$ and this means that $\diam(G(T)) =2$.
\hfill $\diamond$ \end{exm}

Let us recall that a bounded lattice is called dually atomic if for every $x\in S-\{1\}$, there exists a dual atom $m$ such that $a \leq m$ \cite[\S1]{ChajdaHalasKuhr2007}.

\begin{thm}
\label{IntersectionGraphLIdLatticeThm}
Let $(S,+,\cdot,0,1)$ be a dually atomic bounded distributive lattice in which  $\DAtom(S)$ is nonempty.  If the graph $G(S)$ of $S$ has no isolated vertex, then $G(S)$ is connected with $\diam(G(S))\leq 4$.
\end{thm}
\begin{proof}
Let ${a}_1, {a}_2$  be two distinct vertices of $G(S)$. If ${a}_1 {a}_2 \neq 0$, then $d({a}_1,{a}_2)=1$. Now, suppose that ${a}_1 {a}_2=0$. By assumption, there are two elements ${m}_1$ and ${m}_2$ in $\DAtom(S)$ such that ${a}_1\leq {m}_1$ and ${a}_2\leq {m}_2$.

If ${m}_1={m}_2$, or ${m}_1  {m}_2\neq 0$, or ${a}_1 {m}_2\neq 0$, or
${a}_2 {m}_1\neq 0$, then $d({a}_1,{a}_2)\leq 3$. Therefore, we assume that ${m}_1\neq {m}_2$, ${m}_1 {m}_2= 0$, ${a}_1 {m}_2= 0$ and ${a}_2 {m}_1= 0$. If ${a}_1 + {a}_2 \neq 1$, then the path ${a}_1-({a}_1+{a}_2)-{a}_2$ is   of length 2. So, $d({a}_1,{a}_2)\leq 2$.

But if ${a}_1+{a}_2=1$, then ${m}_1={m}_1 {a}_1+{m}_1 {a}_2$. Since ${m}_1 {a}_2=0$, we have ${m}_1={m}_1 {a}_1$, which implies that ${a}_1={m}_1$. By a similar argument, ${a}_2={m}_2$. Also, since $G(S)$ has no isolated vertex, there are two vertices ${b}_1\neq {a}_1$ and ${b}_2\neq {a}_2$ such that ${a}_1$ is adjacent to ${b}_1$ and ${a}_2$ is adjacent to ${b}_2$.

If  ${b}_1  {b}_2 \neq 0$, then ${a}_1-{b}_1-{b}_2-{a}_2$ is a path of length 3. So, $d({a}_1,{a}_2)\leq 3$. If ${a}_1 {b}_2\neq 0$ or ${a}_2 {b}_1\neq 0$, then  $d({a}_1,{a}_2)\leq 2$.

Finally, let ${b}_1 {b}_2= 0$, ${a}_1 {b}_2= 0$, and ${a}_2 {b}_1= 0$. Our claim is that ${b}_1+{b}_2 \neq 1$. In contrary, let ${b}_1+{b}_2=1$. So, ${a}_1 = {a}_1 {b}_1+{a}_1 {b}_2$. So, ${a}_1 = {a}_1 {b}_1$, which implies that ${a}_1 \leq {b}_1$. But ${a}_1 = {m}_1$, so ${a}_1 = {b}_1$, a contradiction. By a similar argument, we have ${a}_2 = {b}_2$. Hence, ${b}_1+{b}_2 \neq 1$ and ${a}_1-{b}_1-({b}_1+{b}_2)-{b}_2-{a}_2$ is a path of length 4. So, $d({a}_1,{a}_2)\leq 4$. Consequently, $G(S)$ is connected with $\diam(G(R))\leq 4$ and the proof is complete.
\end{proof}

Let $R$ be a commutative ring with a nonzero identity and $M$ be a nonzero unital $R$-module. The $R$-module $M$ is called distributive if the lattice of $R$-submodules of $M$ is distributive \cite{Camillo1975}. As in \cite{AkbariTavallaeeGhezelahmad2012}, we denote the intersection graph of submodules of the $R$-module $M$ by $G(M)$.

\begin{cor}

\label{DistributiveModule}

Let $R$ be a commutative ring with a nonzero identity and $M$ be a nonzero unital $R$-module. If $M$ is a distributive $R$-module such that any submodule of $M$ is a subset of a maximal submodule of $M$ and $G(M)$ has no isolated vertex, then $G(M)$ is connected with $\diam(G(M))\leq 4$.
\hfill $\Box$ 
\end{cor}

Let us recall that a trail in a graph $G$ is a walk in which all edges are distinct. A path in the graph $G$ is a trail in which all vertices (except possibly the first and last) are distinct. If $ P = x_0 \cdots x_{k-1}$ is a path in $G$ and $k \geq 3$, then the path $C = x_0 \cdots x_{k-1} x_0$ is a cycle in $G$. The minimum length of a cycle (contained) in the graph $G$ is called the girth of $G$, denoted by $\girth(G(S))$ \cite[p. 8]{Diestel2017}.

\begin{thm}
\label{girth}
Let $S$ be a bounded semilattice. If $G(S)$ contains a cycle, then we have $\girth(G(S)) = 3$.
\begin{proof}
In contrary, suppose that $\girth(G(S)) \geq 4$. This implies that every pair of elements $y$ and $z$ in $S-\{0,1\}$ with $yz \neq 0$ are comparable, because if they are not comparable, then $yz$ is different from $y$ and $z$ and therefore, $y - yz - z - y$ is cycle of length 3, a contradiction. Now, let $z-y-x-t$ be a path of length 3 in $G(S)$. Since any two elements in this path are comparable and any chain of length 2 in $S-\{0,1\}$ induces a cycle of length 3 in $G(S)$, the only possible cases are: $z \leq y$, $x \leq y$, $x \leq t$, or $y \leq z$, $y \leq x$ and we prove that each case leads us to a contradiction.

Case 1: If $z \leq y$, $x \leq y$, and $x \leq t$, then $x \leq yt$, which implies that $yt \neq 0$. Therefore, $y-x-t-x$ is a cycle of length 3 in $G(S)$, a contradiction.

Case 2: If $y \leq z$ and $y \leq x$, then $y \leq xz$, which implies that $xz \neq 0$. Therefore, $z-y-x-z$ is a cycle of length 3 in $G(S)$, again a contradiction. Hence, $\girth(G(S)) = 3$ and the proof is complete.
\end{proof}

\end{thm}

\section*{Acknowledgment}
 The research of the first author was in part supported by a grant from the Islamic Azad University of Qazvin Branch. The second author is supported by the Department of Engineering Science at the Golpayegan University of Technology and his special thanks go to the department for providing all the necessary facilities available to him for successfully conducting this research. The authors are grateful to John LaGrange for looking through the paper and telling some points which improved the paper.

\end{document}